\documentclass[11pt]{amsart}
\usepackage{amsmath,amssymb,amsthm}
\usepackage[mathscr]{eucal}
\numberwithin{equation}{section}
\newtheorem{theorem}{Theorem}[section]
\newtheorem{lemma}[theorem]{Lemma}
\newcommand{\LL}{\mathcal{L}}
\newcommand{\R}{\mathbb{R}}
\newcommand{\T}{\mathbb{T}}
\newcommand{\Z}{\mathbb{Z}}
\newcommand{\dist}{\operatorname{dist}}
\newcommand{\norm}[1]{\left\|#1\right\|}
\allowdisplaybreaks

\begin{document}
\title[Endpoint failure for multilinear rough singular integrals]
{Multilinear rough singular integral operators may not be of the weak type at the endpoint}
\author{Petr Honz\'ik}
\address{Department of Mathematics, Charles University, 116 36 Praha 1, Czech Republic}
\email{honzik@gmail.com}
\thanks{P. Honz\'\i{}k was supported by the grant GA\v{C}R P201/24-10505S}
\date{}

\begin{abstract}
For every $m\ge2$ and $n\ge1$, we construct a mean-zero function
$\Omega\in L^\infty(\mathbb S^{mn-1})$ for which the associated
$m$-linear homogeneous singular integral operator fails to map
$L^1(\R^n)\times\cdots\times L^1(\R^n)$ to $L^{1/m,\infty}(\R^n)$.
More generally, the corresponding weak estimate can fail whenever at
least two input exponents equal one. The proof combines a dimensional
construction with an arithmetic construction in dimension one, followed
by a multilinear lifting argument.
\end{abstract}
\maketitle

\section{Introduction}

Let $\Omega$ be an integrable function on the unit sphere
$\mathbb S^{n-1}$ with mean value zero. The rough singular integral
operator associated with $\Omega$ is defined by
\[
\LL_\Omega f(x)=\operatorname{p.v.}\int_{\R^n}
\frac{\Omega(y/|y|)}{|y|^n}f(x-y)\,dy,
\]
initially for $f$ in the Schwartz class $\mathcal S(\R^n)$.
Calder\'on and Zygmund \cite{Ca_Zy1956} proved that if
$\Omega\in L\log L(\mathbb S^{n-1})$, then $\LL_\Omega$ is bounded
on $L^p(\R^n)$ for all $1<p<\infty$. This result was improved by
Coifman and Weiss \cite{Co_We1977}, who replaced this assumption by
the less restrictive condition $\Omega\in H^1(\mathbb S^{n-1})$.
The same conclusion was obtained independently by Connett \cite{Co1979}.

Analogous questions for multilinear singular integral operators have
been studied by many authors. Let $m\ge2$ and let $\Omega$ be an
integrable function on $\mathbb S^{mn-1}$ with mean value zero. Set
\[
K(\vec y)=\frac{\Omega(\vec y/|\vec y|)}{|\vec y|^{mn}},
\qquad \vec y\ne0.
\]
The associated multilinear singular integral operator is defined by
\[
\LL_\Omega(f_1,\ldots,f_m)(x)
=\operatorname{p.v.}\int_{(\R^n)^m}
K(\vec y)\prod_{j=1}^m f_j(x-y_j)\,d\vec y,
\]
where $x\in\R^n$, $\vec y=(y_1,\ldots,y_m)\in(\R^n)^m$, and the
inputs are initially Schwartz functions.

Coifman and Meyer \cite{Co_Me1975} showed that if $\Omega$ is a
function of bounded variation on the circle $\mathbb S^1$, then
the corresponding bilinear operator is bounded from
$L^{p_1}(\R)\times L^{p_2}(\R)$ to $L^p(\R)$ when
$1<p_1,p_2,p<\infty$ and $1/p=1/p_1+1/p_2$. For general
$n\ge1$ and $m\ge2$, Grafakos and Torres \cite{Gr_To2002}
established the boundedness
\[
\LL_\Omega:L^{p_1}(\R^n)\times\cdots\times L^{p_m}(\R^n)
\longrightarrow L^p(\R^n),
\qquad \frac1p=\sum_{j=1}^m\frac1{p_j},
\]
for all $1<p_j<\infty$ when $\Omega$ is a Lipschitz function
on $\mathbb S^{mn-1}$.

For rough angular functions in the bilinear one-dimensional case,
Diestel, Grafakos, Honz\'ik, Si and Terwilleger \cite{counterex}
obtained estimates by a method of rotations involving the bilinear
Hilbert transform. Grafakos, He and Honz\'ik \cite{Gr_He_Ho2018}
proved the full range of bilinear estimates for bounded angular
functions. The corresponding multilinear result was obtained by
Grafakos, He, Honz\'ik and Park \cite{Gr_He_Ho_Pa2024}.

In the linear setting, weak type $(1,1)$ estimates for rough kernels
were proved in dimension two by Hofmann \cite{Ho1988} under the
assumption $\Omega\in L^s$, $s>1$, and by Christ and Rubio de
Francia \cite{Ch_Ru1988} under the $L\log L$ assumption. Seeger
\cite{Se1996} established the latter result in all dimensions.
For smooth multilinear kernels, weak endpoint estimates follow from
the theory of Grafakos and Torres \cite{Gr_To2002}.
Grafakos, He and Slav\'ikov\'a \cite{Gr_He_Sl2019}
obtained counterexamples to the bilinear weak endpoint estimate
in dimension one with $\Omega\in L^s(\mathbb S^1)$ for any
prescribed $1\le s<\infty$; their result does not cover bounded
angular kernels.

In this note, we show that weak endpoint estimates can fail
even for $\Omega\in L^\infty(\mathbb S^{mn-1})$, in every
dimension and for every degree of multilinearity $m\ge2$.
More precisely, we prove the following.

\begin{theorem}\label{thm:main}
Let $m,n\in\mathbb N$, $m\ge2$, and let
$1\le p_j\le\infty$, $j=1,\ldots,m$, with
\[
\frac1q=\sum_{j=1}^m\frac1{p_j},\qquad \frac1\infty=0.
\]
If $p_1=p_2=1$, then there exists a mean-zero function
$\Omega\in L^\infty(\mathbb S^{mn-1})$ such that
$\LL_\Omega$ is not bounded from
$L^{p_1}(\R^n)\times\cdots\times L^{p_m}(\R^n)$
to $L^{q,\infty}(\R^n)$.
The two $L^1$ factors may occupy any prescribed positions
after permuting the kernel variables.
\end{theorem}

We use the convention
\[
\norm{F}_{L^{q,\infty}}
=\sup_{\lambda>0}\lambda\,
\bigl|\{x:|F(x)|>\lambda\}\bigr|^{1/q}.
\]
All counterexamples below are witnessed by nonnegative smooth
compactly supported inputs. The first two input supports are
separated, so the integrals used in the proof are absolutely
convergent. The proof is given in Sections~\ref{sec:dimensional}
and~\ref{sec:arithmetic}.

\section{Overlap lemmas}\label{sec:overlap}

In order to construct the kernel, we will need the following
lemmas, which control the growth of the sets. Throughout this
section, $\T=\R/\Z$ carries normalized Lebesgue measure.

\begin{lemma}[Growth of inverse-image unions]\label{lem:union-growth}
Let $T:\T\to\T$ be a map such that $T^{-1}A$ is measurable
and $|T^{-1}A|=|A|$ for every Lebesgue measurable set
$A\subset\T$. If $E\subset\T$ is measurable and
$1\le L\le M$ are integers, then
\[
\left|\bigcup_{j=0}^{M-1}T^{-j}E\right|
\le\frac ML\left|\bigcup_{j=0}^{L-1}T^{-j}E\right|.
\]
\end{lemma}

\begin{proof}
Set $V_0=\varnothing$, $V_m=\bigcup_{j=0}^{m-1}T^{-j}E$
and $v_m=|V_m|$. Since $V_{m+1}=E\cup T^{-1}V_m$ and
$|T^{-1}V_m|=|V_m|$, we have
\[
v_{m+1}-v_m
=\left|E\setminus\bigcup_{j=1}^{m}T^{-j}E\right|.
\]
The sets on the right decrease as $m$ increases. Thus the
increments $v_{m+1}-v_m$ are nonincreasing, and consequently
their averages $v_m/m$ are nonincreasing. This gives the assertion.
\end{proof}

Fix an integer $R\ge3$, and let $\pi_1,\ldots,\pi_s$ be the
primes not exceeding $R$. Put
\[
D=\lceil\log_2 R\rceil,\qquad L=2sD,
\]
and define the finite sets of positive integers
\[
A=\left\{\prod_{i=1}^s\pi_i^{a_i}:0\le a_i<L\right\},
\qquad
B=\left\{\prod_{i=1}^s\pi_i^{b_i}:0\le b_i<L+D\right\}.
\]
Every positive odd integer $d\le R$ has all its prime exponents
at most $D$. Hence
\begin{equation}\label{eq:box}
dA\subset B\qquad(d\text{ odd},\ 1\le d\le R).
\end{equation}

\begin{lemma}\label{lem:sets}
For every $0<r<1/4$, there are open sets $J,G\subset\T$,
each a finite union of intervals with rational endpoints, such that
\[
\frac r2<|J|\le r,\qquad |G|\le4r,
\]
and
\begin{equation}\label{eq:exact}
\begin{gathered}
x\bmod1\in\overline J,\qquad d\text{ odd},\quad0<|d|\le R,
\qquad j\in\Z
\\
\Longrightarrow\quad \frac{x-j}{d}\bmod1\in G.
\end{gathered}
\end{equation}
Here $x\in\R$, and the closure is taken in $\T$.
\end{lemma}

\begin{proof}
Write $\|v\|_{\T}=\dist(v,\Z)$. For a finite set $C$ of
positive integers and $0<t\le1$, let
\[
H_C(t)=\bigcup_{a\in C}\{x\in\T:\|ax\|_{\T}<t/2\},
\qquad F_C(t)=|H_C(t)|.
\]
We first show that
\begin{equation}\label{eq:union-growth}
F_B(t)\le\left(1+\frac DL\right)^sF_A(t)<2F_A(t).
\end{equation}
For each $i$, the map $T_i(x)=\pi_i x\bmod1$ satisfies the
hypotheses of Lemma~\ref{lem:union-growth}. Indeed, identifying
$\T$ with $[0,1)$, for every measurable $U\subset[0,1)$ we have
\[
T_i^{-1}U=\bigcup_{j=0}^{\pi_i-1}\frac{U+j}{\pi_i}.
\]
These sets are measurable and disjoint, lie in the intervals
$[j/\pi_i,(j+1)/\pi_i)$, and each has measure $|U|/\pi_i$.
Thus $|T_i^{-1}U|=|U|$. The maps $T_i$ commute, so we can
apply Lemma~\ref{lem:union-growth} successively in each coordinate,
enlarging its exponent range from $L$ to $L+D$. This gives the
first inequality in \eqref{eq:union-growth}. The second follows from
\[
\left(1+\frac DL\right)^s
\le\exp\left(\frac{sD}{L}\right)=\sqrt e<2.
\]

We also use the elementary estimate
\begin{equation}\label{eq:interval-doubling}
F_C(2t)\le2F_C(t),\qquad0<t\le1/2.
\end{equation}
The intervals defining $H_C(2t)$ are obtained by doubling about
their centers the intervals defining $H_C(t)$. If $F_C(t)=1$,
the assertion is immediate. Otherwise every component of $H_C(t)$
is a proper arc, and the doubled intervals belonging to this
component lie in the concentric doubled arc. Summing the lengths
over the components proves \eqref{eq:interval-doubling}.

Choose the least nonnegative integer $\ell$ such that
$F_A(2^{-\ell})\le r$, and set $\eta=2^{-\ell}$. Such an
integer exists because $F_A(t)\le |A|t\to0$ as $t\to0$.
Since $F_A(1)=1$, we have $\ell\ge1$, and the minimality of
$\ell$, together with \eqref{eq:interval-doubling}, gives
\[
r<F_A(2\eta)\le2F_A(\eta).
\]
Thus $r/2<F_A(\eta)\le r$. Define
\[
J=H_A(\eta),\qquad G=H_B(2\eta).
\]
By \eqref{eq:union-growth} and \eqref{eq:interval-doubling},
\[
|G|<2F_A(2\eta)\le4F_A(\eta)\le4r.
\]

If $x\bmod1\in\overline J$, then for some $a\in A$ we have
$\|ax\|_{\T}\le\eta/2$. Let $d$ and $j$ be as in
\eqref{eq:exact}, and put $v=(x-j)/d$. By \eqref{eq:box},
$b=a|d|\in B$, and
\[
bv=\operatorname{sgn}(d)a(x-j),\qquad
\|bv\|_{\T}=\|ax\|_{\T}\le\eta/2<\eta.
\]
Consequently $v\bmod1\in G$, including when
$x\bmod1$ lies on the boundary of $J$. Finally, the endpoints
of the intervals are rational because $\eta$ is dyadic. When
$r$ is rational, the choice of $\ell$ requires only finite exact
calculations with unions of intervals.
\end{proof}

\section{Dimensional construction}\label{sec:dimensional}

When $n\ge2$, or when $n=1$, $m\ge3$ and $p_3=1$, we can
use simple dimensional considerations to prove
Theorem~\ref{thm:main}.

For an integer $N\ge1$, set
\[
A_N=\Z^n\cap[-2N,2N]^n,\qquad B_N=A_N+\tfrac12 e_1,
\]
where $e_1$ is the first coordinate vector. Define
\[
C_N=
\begin{cases}
A_N\times B_N\times(\R^n)^{m-2},& n\ge2,\\
A_N\times B_N\times A_N\times\R^{m-3},& n=1.
\end{cases}
\]
An empty product of free factors is omitted. Write $C_N$ as a
finite union of affine subspaces $Q_{N,\nu}$, obtained by fixing
the lattice coordinates and leaving the remaining coordinates free.
Let
\[
\Delta=\{(x,\ldots,x):x\in\R^n\},\qquad
D_N=\bigcup_\nu\operatorname{span}(\Delta\cup Q_{N,\nu}).
\]
If $Q_{N,\nu}=c_\nu+V$, then
\[
\operatorname{span}(\Delta\cup Q_{N,\nu})
=\Delta+V+\operatorname{span}\{c_\nu\}.
\]
Its dimension is therefore at most
\[
\dim\Delta+\dim Q_{N,\nu}+1
=\begin{cases}
(m-1)n+1\le mn-1,& n\ge2,\\
m-1=mn-1,& n=1.
\end{cases}
\]
Thus $D_N$ is a finite union of proper linear subspaces, and
\[
E_N=\mathbb S^{mn-1}\cap D_N
\]
is closed and has surface measure zero. Notice also that
$(x,\ldots,x)-z\in D_N$ for every $x\in\R^n$ and $z\in C_N$.

Write $\sigma_{mn-1}$ for surface measure on $\mathbb S^{mn-1}$.
Choose $0<\varepsilon(N)\le1$ so that the open neighborhood
\[
F_N=\bigl(E_N+B(0,\varepsilon(N))\bigr)\cap\mathbb S^{mn-1}
\]
has measure at most $2^{-N-20}\sigma_{mn-1}(\mathbb S^{mn-1})$.
Put
\[
G=\bigcup_{N=20}^\infty F_N,\qquad
\Omega=\mathbf1_G-\alpha\mathbf1_{\mathbb S^{mn-1}\setminus G},
\]
where
\[
\alpha=\frac{\sigma_{mn-1}(G)}
{\sigma_{mn-1}(\mathbb S^{mn-1})-\sigma_{mn-1}(G)}.
\]
Then $\Omega$ has mean value zero. Moreover,
$\sigma_{mn-1}(G)\le2^{-39}\sigma_{mn-1}(\mathbb S^{mn-1})$,
so $\alpha<1$ and $\norm{\Omega}_\infty=1$.

Take a nonnegative function $\varphi\in C_c^\infty(B(0,1))$
with integral one, and a function
$\psi\in C_c^\infty((-2,2)^n)$ with $0\le\psi\le1$ and
$\psi=1$ on $[-1,1]^n$. For sufficiently large $N$, put
$\delta_N=\varepsilon(4N)/100$ and set
\[
f_{1,N}(u)=\sum_{a\in A_N}\delta_N^{-n}
\varphi\left(\frac{u-a}{\delta_N}\right),\qquad
f_{2,N}(u)=\sum_{b\in B_N}\delta_N^{-n}
\varphi\left(\frac{u-b}{\delta_N}\right).
\]
In the case $n=1$, $m\ge3$, also set $f_{3,N}=f_{1,N}$.
For every remaining input, set $f_{j,N}(u)=\psi(u/N)$.
The lattice inputs have $L^1$ norm $(4N+1)^n\asymp N^n$,
and the other inputs satisfy
\[
\norm{f_{j,N}}_{p_j}=N^{n/p_j}\norm{\psi}_{p_j},
\]
with the usual interpretation when $p_j=\infty$.

We next verify that all interactions are positive for
$x\in[-N,N]^n$. Given a point $s=(s_1,\ldots,s_m)$ in the
product of the input supports, choose $z\in C_N$ by replacing
each lattice input coordinate by its corresponding lattice center
and keeping the free coordinates unchanged. Since there are at
most three lattice inputs, $|s-z|\le\sqrt3\delta_N$. Let
\[
v=(x-z_1,\ldots,x-z_m),\qquad
w=(x-s_1,\ldots,x-s_m).
\]
The lattice separation $|z_1-z_2|\ge1/2$ implies
$|v|\ge1/(2\sqrt2)$, and hence
\[
\left|\frac{w}{|w|}-\frac{v}{|v|}\right|
\le\frac{2|w-v|}{|v|}
\le4\sqrt6\delta_N<\varepsilon(4N).
\]
Since $v/|v|\in E_N\subset E_{4N}$, we have
$w/|w|\in F_{4N}\subset G$. The first two input supports
are separated, so the integral is absolutely convergent and
\[
\LL_\Omega(f_{1,N},\ldots,f_{m,N})(x)
=\int_{(\R^n)^m}
\frac{\prod_{j=1}^m f_{j,N}(s_j)}
{|(x-s_1,\ldots,x-s_m)|^{mn}}\,d\vec s.
\]

To obtain a logarithmic lower bound, consider the scales
$t=4^\ell$ with $4\le t\le N/4$. For the first lattice
input, select centers $a$ satisfying
\[
t\le a_1-x_1\le2t,\qquad |a_i-x_i|\le t\quad(2\le i\le n).
\]
For every other lattice input, select centers at coordinatewise
distance at most $t$ from $x$. These centers belong to the
relevant sets $A_N$ or $B_N$, and each selection contains at
least $c_n t^n$ points. For each free input, restrict $s_j$ to
\[
[-N,N]^n\cap\bigl(x+[-t,t]^n\bigr),
\]
which has volume at least $t^n$ and on which $f_{j,N}=1$.
On this region,
$|(x-s_1,\ldots,x-s_m)|\le C_{m,n}t$.
Each lattice bump has integral one, so the contribution from
each scale is bounded below by a positive constant independent
of $N$ and $x$. The ranges of the first coordinate of the first
input are disjoint for distinct scales, even after thickening
by $\delta_N$. Summing the contributions gives
\[
\LL_\Omega(f_{1,N},\ldots,f_{m,N})(x)\gtrsim\log N,
\qquad x\in[-N,N]^n.
\]
Consequently,
\[
\norm{\LL_\Omega(f_{1,N},\ldots,f_{m,N})}_{L^{q,\infty}}
\gtrsim N^{n/q}\log N,
\qquad
\prod_{j=1}^m\norm{f_{j,N}}_{p_j}\asymp N^{n/q}.
\]
The ratio tends to infinity. This proves the theorem in the
cases considered in this section.

\section{Arithmetic construction}\label{sec:arithmetic}

In the remaining cases we use a more complicated arithmetic
construction in dimension one. We first construct a bilinear
counterexample, and then lift it to the multilinear setting.
In fact, this construction works for every $m\ge2$ on $\R$
whenever two prescribed inputs belong to $L^1$.

We write $\sigma$ for arc length measure on $\mathbb S^1$.
For $y=(y_1,y_2)$ with $y_1\ne y_2$, define
\begin{equation}\label{eq:coordinate}
u(y)=\frac{y_1}{2(y_1-y_2)}.
\end{equation}
This coordinate is homogeneous of degree zero and satisfies
$u(-y)=u(y)$.

To explain the role of the arithmetic construction, consider
the pairs $(a,a+1/2)$, where $a\in\Z\cap[-N,N]$.
As $x$ ranges over $[-N,N]$, the coordinate
$u(x-a,x-a-1/2)=x-a$ fills $[-2N,2N]$.
The pairs $(a,a-1/2)$ give the opposite semicircle.
Thus requiring positivity for all $x\in[-N,N]$ covers all
but arcs of length $O(1/N)$. A two-valued mean-zero angular
function normalized to have $L^\infty$ norm one would then
have a positive value tending to zero. We instead use periodic
output sets of small controlled density and compensate for this
loss by the logarithmic gain on a much larger spatial lattice.

\subsection{One fixed angular kernel}

For $k\ge1$, let
\begin{equation}\label{eq:scales}
r_k=2^{-k-10},\qquad N_k=2^{2^{k^2}},\qquad R_k=12N_k+1.
\end{equation}
For every fixed $q>0$, these choices give
\begin{equation}\label{eq:growth}
r_k^{1/q}\log N_k
=(\log2)\,2^{k^2-(k+10)/q}\longrightarrow\infty.
\end{equation}
Apply Lemma~\ref{lem:sets} with $R=R_k$ and $r=r_k$,
obtaining sets $J_k,G_k$. Set
$G_* =\bigcup_{k\ge1}G_k$. Then
\begin{equation}\label{eq:small}
|G_*|\le4\sum_{k\ge1}r_k=\frac1{256}.
\end{equation}
Define the open antipodally symmetric set
\[
\Gamma_0=\{\theta\in\mathbb S^1:\theta_1\ne\theta_2,
\ u(\theta)\bmod1\in G_*\}.
\]
One of the two semicircles complementary to the diagonal is
parametrized by
\[
\theta(u)=\frac{(2u,2u-1)}{\sqrt{(2u)^2+(2u-1)^2}},
\qquad u\in\R.
\]
Including the opposite semicircle and periodizing gives
\begin{equation}\label{eq:weight}
\begin{aligned}
\sigma(\Gamma_0)&=\int_0^1\mathbf1_{G_*}(u)W(u)\,du,\\
W(u)&=\sum_{j\in\Z}\frac8{1+16(u+j-1/4)^2}\le13.
\end{aligned}
\end{equation}
Indeed, the term corresponding to the nearest integer is at
most $8$, and the remaining terms have sum at most
\[
\sum_{j=1}^\infty\frac1{(j-1/2)^2}=\frac{\pi^2}{2}.
\]
Thus $W(u)\le8+\pi^2/2<13$. By \eqref{eq:small},
$\sigma(\Gamma_0)\le13/256<\pi$.
Enlarge $\Gamma_0$ to a measurable antipodally symmetric set
$\Gamma$ of arc length exactly $\pi$. This is possible by
adding a measurable subset of its complement in one semicircle
and the antipodal image of that subset. Set
\begin{equation}\label{eq:Omega}
\Omega=\mathbf1_\Gamma-\mathbf1_{\mathbb S^1\setminus\Gamma},
\qquad K(y)=\frac{\Omega(y/|y|)}{|y|^2}.
\end{equation}
This is one fixed kernel, with $\norm{\Omega}_\infty=1$
and $\int_{\mathbb S^1}\Omega\,d\sigma=0$. Its angular
function equals one on every direction in $\Gamma_0$.

\subsection{Smooth lattice tests and the bilinear estimate}

Fix $k$, write $N=N_k$, and let
\[
A_N=\Z\cap[-3N,3N],\qquad B_N=A_N+\tfrac12,
\qquad E_k=\{x\in[-N,N]:x\bmod1\in J_k\}.
\]
Since $N$ is an integer,
\begin{equation}\label{eq:E}
|E_k|=2N|J_k|\ge r_kN.
\end{equation}
For $a\in A_N$ and $b\in B_N$, the integer $d=2(b-a)$
is odd and satisfies $1\le|d|\le12N+1=R_k$. Moreover,
\[
u(x-a,x-b)=\frac{x-a}{d}.
\]
Lemma~\ref{lem:sets}, with $j=a$, shows that the direction
of $(x-a,x-b)$ lies in $\Gamma_0$ whenever
$x\in\overline{E_k}$. The directions corresponding to this
compact output set and the finitely many lattice pairs form a
compact subset of the open set $\Gamma_0$. Also,
$|a-b|\ge1/2$, so the vectors stay away from zero.
Consequently, there is a common $0<\delta_k<1/20$ such that
\[
\frac{(x-s,x-t)}{|(x-s,x-t)|}\in\Gamma_0
\]
whenever $x\in\overline{E_k}$, $(a,b)\in A_N\times B_N$,
$|s-a|\le\delta_k$, and $|t-b|\le\delta_k$.

Take a nonnegative $\varphi\in C_c^\infty((-1,1))$ with
integral one, and define
\[
f_k(s)=\sum_{a\in A_N}\delta_k^{-1}
\varphi\left(\frac{s-a}{\delta_k}\right),\qquad
g_k(s)=\sum_{b\in B_N}\delta_k^{-1}
\varphi\left(\frac{s-b}{\delta_k}\right).
\]
Then $\norm{f_k}_1=\norm{g_k}_1=6N+1$, independently
of $\delta_k$. Their supports are separated by at least
$1/2-2\delta_k>2/5$, and hence the defining integrals are
absolutely convergent. By the choice of $\delta_k$,
\begin{equation}\label{eq:positive}
\LL_\Omega(f_k,g_k)(x)
=\int_{\R^2}\frac{f_k(s)g_k(t)}{(x-s)^2+(x-t)^2}\,ds\,dt,
\qquad x\in E_k.
\end{equation}

For every positive odd integer $d\le N$, choose the integers
$a\in[x-d,x]$ and put $b=a+d/2$. There are at least $d$
such integers $a$. Since $x\in[-N,N]$, we have
$a\in[-2N,N]$ and
$b-1/2=a+(d-1)/2\in[-2N,2N]$; thus all these pairs
belong to $A_N\times B_N$. Throughout their bumps,
\[
(x-s)^2+(x-t)^2
\le(d+\delta_k)^2+(d/2+\delta_k)^2\le2d^2.
\]
Every product of two bumps has integral one, and therefore
contributes at least $1/(2d^2)$ to \eqref{eq:positive}.
We obtain
\begin{equation}\label{eq:main}
\LL_\Omega(f_k,g_k)(x)
\ge\frac12\sum_{\substack{1\le d\le N\\d\text{ odd}}}\frac1d
\ge\frac14\log N,\qquad x\in E_k.
\end{equation}
Using the threshold $\lambda=(\log N_k)/8$ and \eqref{eq:E},
we conclude that
\[
\frac{\norm{\LL_\Omega(f_k,g_k)}_{L^{1/2,\infty}}}
{\norm{f_k}_1\norm{g_k}_1}
\ge\frac{(\log N_k)(r_kN_k)^2}{8(6N_k+1)^2}
\ge\frac{r_k^2\log N_k}{392}\longrightarrow\infty,
\]
where the last step follows from \eqref{eq:growth} with
$q=1/2$. This proves Theorem~\ref{thm:main} for $m=2$, $n=1$.

\subsection{Multilinear lifting}

Let $m\ge3$, let $1\le p_3,\ldots,p_m\le\infty$, and set
\[
\frac1q=2+\sum_{j=3}^m\frac1{p_j}.
\]
Choose a nonnegative function
$\Phi\in C_c^\infty((-1,1)^{m-2})$ with integral one.
For $y=(y_1,y_2)$, $z=(z_3,\ldots,z_m)$ and $\rho=|y|>0$,
define
\[
\widetilde K_m(y,z)
=K(y)\rho^{-(m-2)}\Phi(z/\rho)
=\Omega(y/\rho)\rho^{-m}\Phi(z/\rho),
\]
and set it equal to zero when $y=0$. This kernel is homogeneous
of degree $-m$; write $\widetilde\Omega_m$ for its angular
function on $\mathbb S^{m-1}$. On the support of the kernel,
$|z|^2\le(m-2)\rho^2$. Thus on the unit sphere,
$\rho\ge(m-1)^{-1/2}$, which shows that
$\widetilde\Omega_m$ is bounded.

To verify cancellation, use the change of variables
$y=\rho\theta$, $z=\rho v$, whose Jacobian is
$\rho^{m-1}\,d\rho\,d\sigma(\theta)\,dv$. Then
\begin{align*}
(\log2)\int_{\mathbb S^{m-1}}\widetilde\Omega_m
&=\int_{1<|(y,z)|<2}\widetilde K_m(y,z)\,dy\,dz\\
&=(\log2)\left(\int_{\mathbb S^1}\Omega\,d\sigma\right)
\left(\int_{\R^{m-2}}\Phi(v)\,dv\right)=0.
\end{align*}
Indeed, for each fixed $v$ the radial interval is
\[
(1+|v|^2)^{-1/2}<\rho<2(1+|v|^2)^{-1/2},
\]
and the radial integral of $d\rho/\rho$ equals $\log2$.

Let $\LL_{\widetilde\Omega_m}$ be the associated multilinear
operator. Choose $\chi\in C_c^\infty((-9,9))$ such that
$0\le\chi\le1$ and $\chi=1$ on $[-8,8]$, and set
\[
h_{j,k}(s)=\chi(s/N_k),\qquad j=3,\ldots,m.
\]
If $|x|\le N_k$ and $f_k(x-y_1)g_k(x-y_2)\ne0$, then
$\rho\le7N_k$. On the support of $\Phi(z/\rho)$ we have
$|z_j|<\rho$, so $|x-z_j|\le8N_k$ and all the additional
input factors equal one. Integrating in $z$ therefore gives
\[
\LL_{\widetilde\Omega_m}(f_k,g_k,h_{3,k},\ldots,h_{m,k})(x)
=\LL_\Omega(f_k,g_k)(x),\qquad |x|\le N_k.
\]
The integrals are absolutely convergent because the first two
input supports are separated. Finally,
\[
\norm{h_{j,k}}_{p_j}=N_k^{1/p_j}\norm{\chi}_{p_j},
\]
with the same convention for $p_j=\infty$. Combining
\eqref{eq:E} and \eqref{eq:main}, we obtain
\begin{multline*}
\frac{\norm{\LL_{\widetilde\Omega_m}
(f_k,g_k,h_{3,k},\ldots,h_{m,k})}_{L^{q,\infty}}}
{\norm{f_k}_1\norm{g_k}_1\prod_{j=3}^m\norm{h_{j,k}}_{p_j}}
\\
\gtrsim r_k^{1/q}\log N_k\longrightarrow\infty.
\end{multline*}
Here the implicit constant is independent of $k$, and the
divergence follows from \eqref{eq:growth}. Permuting the kernel
variables gives the assertion for any prescribed pair of $L^1$
inputs. This completes the proof of Theorem~\ref{thm:main}.
\hfill$\square$

\section{Notes}

It remains an interesting question whether the weak-type
estimate holds when exactly one of the input exponents equals
one. There are several other bilinear and multilinear operators
for which endpoint weak-type estimates remain open, such as
the bilinear Hilbert transform.

\section*{Acknowledgement}

The arithmetic construction in Section~\ref{sec:arithmetic}
was found with the help of ChatGPT.

\end{document}